\documentclass[a4paper,11pt]{article}

\usepackage{amsmath}
\usepackage{latexsym,amsmath}
\usepackage{amsthm}
\usepackage{amssymb}
\usepackage{enumerate}

\usepackage{epsfig, verbatim}
\usepackage{version}
\usepackage{graphicx}
\usepackage{color}

\newtheorem{theorem}{Theorem}

\newtheorem{conjecture}[theorem]{Conjecture}

\begin{document}
\title{Scott's induced subdivision conjecture for maximal triangle-free graphs}

    \author{Nicolas Bousquet$^1$, St\'ephan Thomass\'e$^2$ \\
           \small $^1$ Universit\'e Montpellier 2 - CNRS, LIRMM, \\
               \small  161 rue Ada, 34392 Montpellier, France \\
                \small bousquet@lirmm.fr \\
         \small   $^2$ Laboratoire LIP (U. Lyon, CNRS, ENS Lyon, INRIA, UCBL), \\
               \small  46 Allée d’Italie, 69364 Lyon Cedex 07, France. \\
               \small  stephan.thomasse@ens-lyon.fr}

\maketitle

\begin{abstract}
Scott conjectured in \cite{Scott97} that the class of graphs with no induced subdivision of a given graph is $\chi$-bounded. We verify his conjecture for maximal triangle-free graphs.
\end{abstract}

Let $F$ be a graph. We denote by Forb$^*$($F$) the class of graphs with no 
induced subdivision of $F$. A class $\mathcal G$
of graphs is {\it $\chi$-bounded} if there exists a function $f$ such that every graph $G$ of 
$\mathcal G$ satisfies $\chi(G)\leq f(\omega (G))$, where $\chi$ and $\omega$ respectively denote 
the chromatic number and the clique number of $G$. Gy\'arf\'as conjectured that Forb$^*$($F$) is 
$\chi$-bounded if $F$ is a cycle \cite{gyarfas87}. Scott proved that 
for each tree $T$, Forb$^*$($T$) is $\chi$-bounded and conjectured the following \cite{Scott97}.

\begin{conjecture}\label{scott}
For every graph $F$, Forb$^*$($F$) is $\chi$-bounded.
\end{conjecture}

This question is open for triangle-free graphs, which is probably the core of the problem. 
It also has nice corollaries, for instance it would imply that any collection of 
segments in the plane with no three of them pairwise intersecting can be 
partitioned into a bounded number of non intersecting sets of segment. 
This is a well-known question of Erd\H{o}s, first cited in \cite{gyarfas87}.
Our goal is to prove Scott's 
conjecture for triangle-free graphs with diameter two, i.e. maximal triangle-free graphs.

\begin{theorem}\label{2scott}
Let $F$ be a graph of size $l$. Every maximal triangle-free graph $G$ with 
$\chi(G) \geq {\mathrm e}^{\theta(l^4)}$ contains an induced subdivision of $F$.
\end{theorem}

\begin{proof}
Let $H$ be the \emph{neighborhood hypergraph} of $G$, i.e. the hypergraph with vertex set $V$ 
and with hyperedges the closed neighborhoods of the vertices of $G$.
Observe that $H$ has packing number one, i.e. its hyperedges pairwise 
intersect. Note also that if the transversality of $H$ is $t$ (minimum size 
of a set of vertices intersecting all hyperedges), then $\chi(G)\leq 2t$. Indeed, $G$
can be covered by $t$ closed neighborhoods, hence by $t$ induced stars since 
$G$ is triangle-free. Since $\chi(G) \geq {\mathrm e}^{\theta(l^4)}$, the transversality of $H$ is at least 
${\mathrm e}^{\theta(l^4)}$.

Ding, Seymour and Winkler \cite{DingSW94} proved that if a hypergraph $H$ has packing number one and transversality 
greater than $11 d^2 (d+4) (d+1)^2$, it contains $d$ hyperedges $e_1,\dots ,e_d$ and a set of vertices 
$Y=\{y_{i,j}~:~1\leq i<j\leq d\}$ such that $y_{i,j}\in e_i\cap e_j$ and $y_{i,j}\notin e_k$ for all 
$k\neq i,j$. Since the transversality of $H$ is at least 
${\mathrm e}^{\theta(l^4)}$, we have such a collection of hyperedges $e_1,\dots ,e_d$ 
with $d\geq {\mathrm e}^{\theta(l^4)}$.

Each $e_i$ corresponds to the closed neighborhood of some vertex $x_i$ of $G$. Let $X=\{x_1,\dots ,x_d\}$. 
By a theorem of Kim \cite{Kim95}, every triangle-free graph on $n$ vertices has a stable 
set of size $\theta(\sqrt{n\log(n)})$. Hence there exists a stable set $S$ in $X$ of size at 
least $\sqrt{d}$ (which is still at least ${\mathrm e}^{\theta(l^4)}$). Free to restrict $X$ 
to $S$, we can assume that $X$ is indeed a stable set, still denoting it by $\{x_1,\dots ,x_d\}$. Note that 
no $y_{i,j}$ belongs to $X$ since $y_{i,j}$ would be a neighbor of some vertex of $X$, 
which is a stable set. Consequently, $X$ is a stable set of 
$G$, the set $Y$ is disjoint from $X$, and for every pair 
$x_i,x_j$ there is a unique vertex $y_{i,j}$ of $Y$ which is joined to 
exactly these two vertices of $X$.

Since the restriction of $G$ to $Y$ is triangle-free and has size $d\choose 2$, by Kim's theorem, 
it contains a stable set $Y'$ of size $\theta(d \sqrt{\log(d)})$. 

Consider the graph $G'$ on vertex set $X$ with an edge $x_ix_j$ if and only if $y_{i,j} \in Y'$. 
Note that if $G''$ is a subgraph of $G'$ on vertex set $X'$, then $G''$ appears 
as an induced subdivision in $G$. Indeed, the induced restriction of $G$ to $X'\cup Y''$,
where $y_{i,j} \in Y''$ whenever $x_ix_j$ is an edge of $G''$, is such a 
subdivision. So we just have to show that $G''$ contains a subdivision of our 
original graph $F$ as a subgraph, which is granted by the last result.

A theorem due to Mader \cite{Mader67} and improved by Bollob\'as and Thomason \cite{BollobasT98} 
ensures that each graph with average degree $512 l^2$ contains a subdivision of 
$K_l$, hence of $F$. Since $d \geq {\mathrm e}^{\theta(l^4)}$, we have $\sqrt{\log(d)} \geq 256 l^2$, 
hence $G'$ contains a subdivision of $F$, and therefore $G$ has an induced subdivision of $F$.
\end{proof}

Ding, Seymour and Winkler also bound the transversality of $H$ in terms of the packing number.
Hence the chromatic number of $G$ is also bounded when the maximum packing of neighborhoods is bounded.
This is the case for instance if the minimum degree is $c.n$ for some fixed constant $c>0$.

\bibliographystyle{plain}
\bibliography{bibli}

\end{document}